\documentclass[12pt]{amsart}
\usepackage[all, knot]{xy}
\usepackage{epsfig}
\usepackage{xspace}
\usepackage{layout}
\usepackage{latexsym}
\usepackage{amsmath}
\usepackage{amsthm}
\usepackage{amssymb}
\usepackage{amsfonts}
\usepackage{amsxtra}     
\usepackage{color}
\usepackage{verbatim}
\usepackage{longtable}
\usepackage{url}
\usepackage{tikz}
\usetikzlibrary{matrix,arrows,decorations.pathmorphing}

\newcommand{\co}{\colon\thinspace}

\begin{document}

\newtheorem{theorem}{Theorem}[section]
\newtheorem{conj}[theorem]{Conjecture}
\newtheorem{lem}[theorem]{Lemma}
\newtheorem{corollary}[theorem]{Corollary}
\newtheorem{proposition}[theorem]{Proposition}
\newtheorem{rem}[theorem]{Remark}

\theoremstyle{definition}
\newtheorem{mydefinition}[theorem]{Definition}
\newtheorem{myexample}[theorem]{Example}
\newtheorem{construction}[theorem]{Construction}
\newtheorem{notation}[theorem]{Notation}
\newtheorem{myremark}[theorem]{Remark}

\theoremstyle{remark}

\makeatletter
\renewcommand{\maketag@@@}[1]{\hbox{\m@th\normalsize\normalfont#1}}%
\makeatother

\renewcommand{\labelenumi}{(\roman{enumi})}
\renewcommand{\labelenumii}{(\alph{enumii})}

\renewcommand{\theenumi}{(\roman{enumi})}
\renewcommand{\theenumii}{(\alph{enumii})}

\def\square{\hfill ${\vcenter{\vbox{\hrule height.4pt \hbox{\vrule width.4pt
height7pt \kern7pt \vrule width.4pt} \hrule height.4pt}}}$}

\newenvironment{pf}{{\it Proof:}\quad}{\square \vskip 12pt}

\title[Permutations and Duality]{Incorporating Voice Permutations into the Theory of Neo-Riemannian Groups and Lewinian Duality}
\author[Fiore, Noll, Satyendra]{Thomas M. Fiore, Thomas Noll, and Ramon Satyendra}
\address{Thomas M. Fiore \\ Department of Mathematics and Statistics\\
University of Michigan-Dearborn \\ 4901 Evergreen Road \\ Dearborn,
MI 48128 \\ U.S.A.} \email{tmfiore@umich.edu}
\urladdr{http://www-personal.umd.umich.edu/~tmfiore/}
\address{Thomas Noll \\ Escola Superior de M\'{u}sica de Catalunya \\
Departament de Teoria, Composici\'{o} i Direcci\'{o} \\
C. Padilla, 155 - Edifici L'Auditori \\
08013 Barcelona, Spain }
\email{noll@cs.tu-berlin.de}
\urladdr{http://user.cs.tu-berlin.de/~noll/}
\address{Ramon Satyendra \\ School of Music, Theatre and Dance \\
University of Michigan \\ 1100 Baits Drive \\ Ann Arbor,
MI 48109-2085 \\ U.S.A.} \email{ramsat@umich.edu}
\urladdr{http://ramonsatyendra.net/}
\maketitle

\begin{abstract}
A familiar problem in neo-Riemannian theory is that the $P$, $L$, and $R$ operations defined as contextual inversions on pitch-class segments do not produce parsimonious voice leading. We incorporate permutations into $T/I$-$PLR$-duality to resolve this issue and simultaneously broaden the applicability of this duality.  More precisely, we construct the dual group to the permutation group acting on $n$-tuples with distinct entries, and prove that the dual group to permutations adjoined with a group $G$ of invertible affine maps $\mathbb{Z}_{12} \to \mathbb{Z}_{12}$ is the internal direct product of the dual to permutations and the dual to $G$.
Musical examples include Liszt,
R. W. Venezia, S. 201 and Schoenberg, String Quartet Number 1, Opus 7. We also prove that the Fiore--Noll construction of the dual group in the finite case works, and clarify the relationship of permutations with the $\mathrm{RICH}$ transformation.
\\
\\
Keywords: dual group, duality, Lewin, neo-Riemannian group, $PLR$, permutation, RICH, retrograde inversion enchaining
\end{abstract}

\section{Introduction: neo-Riemannian Groups and Voice Leading Parsimony}

In the context of this article we wish to touch a sore spot at the very heart of neo-Riemannian theory. It concerns the remarkable solidarity between voice leading parsimony on the one hand and triadic transformations on the other. How do the two aspects fit together, precisely? The study of voice leading requires the localization of chord tones within an ensemble of voices. The study of triadic transformations, and in particular the investigation of the duality between the $T/I$ and $PLR$-groups, seems either to require an abstraction of the triads from their concrete construction from tones or it leads to a dualistic voice leading behavior, which is in conflict with the principle of voice-leading parsimony (see Figure \ref{fig: CakeCognac}).

\begin{figure}
\centering
\includegraphics[width=11cm]{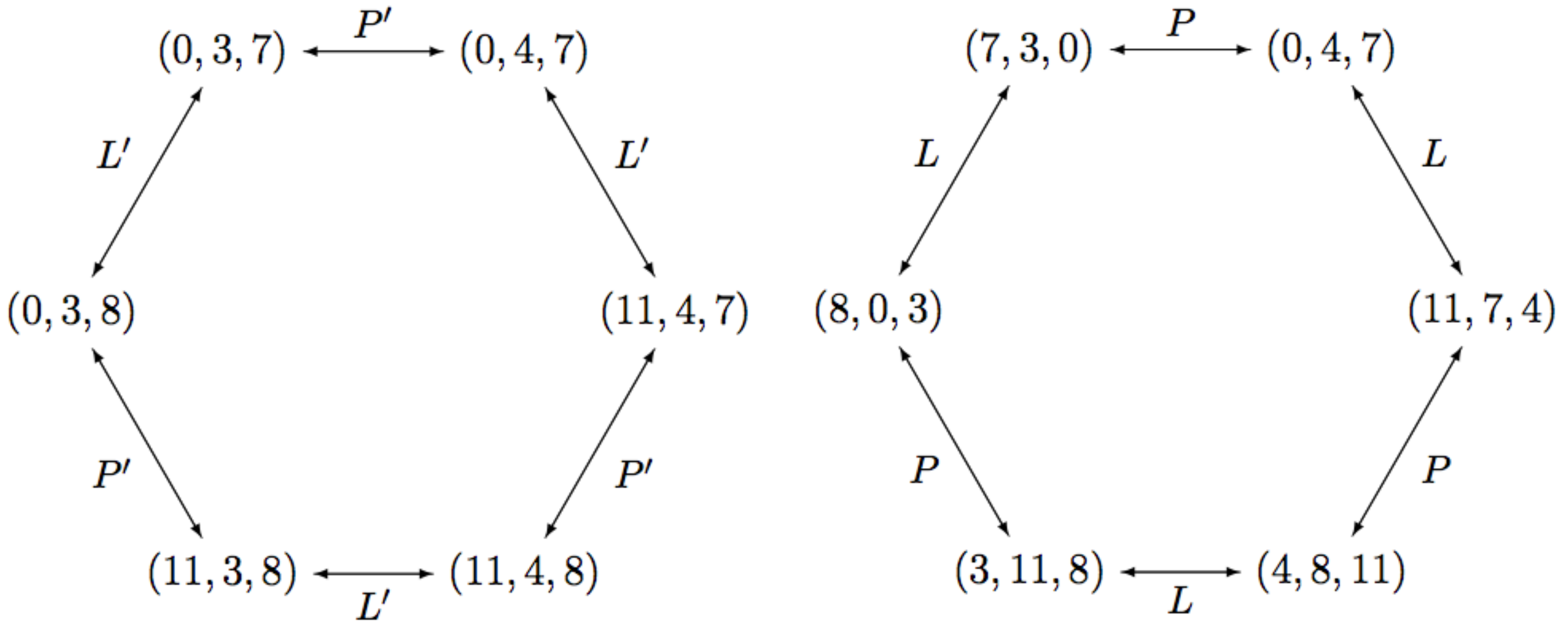}
\caption{Two ``proto-transformational" networks representing different voice-leadings for a Hexatonic Cycle (right: parsimonious voice leading, left: dualistic voice leading).}
\label{fig: CakeCognac}
\end{figure}

In the light of the impact of dialectics upon the development of music theoretical ideas in the writings of Moritz Hauptmann and Hugo Riemann it is remarkable that Nora Engebretsen portraits in \cite{engebretsen} a main line of conceptual development in the second half of the 19th century within the garb of a dialectical triad: \begin{enumerate}
\item Hauptmann's focus on common-tone retention in (diatonic) triadic progressions (Thesis)
\item Von Oettingen's focus on the dualism between major and minor triads (Antithesis)
\item Riemann's attempts to integrate both view points in a chromatic context (Synthesis)
\end{enumerate}

Despite of its historical attractiveness this dialectical metaphor remains euphemistic, until a successful neo-Riemannian synthesis of voice leading and Lewinian transformational theory has been achieved. The present paper takes a step in this direction and, in particular, attributes precise transformational meanings to the arrow labels in the networks of Figure \ref{fig: CakeCognac}.

\section{Construction of the Dual Group in the Finite Case}

In preparation for our treatment of permutations in neo-Riemannian groups, we briefly recall the well-known duality between the $T/I$-group and $PLR$-group, and present a new proof of the Fiore--Noll construction of the dual group in the finite case. The basic objects upon which the $T/I$-group and $PLR$-group act are pitch-class segments with three constituents. Recall that a {\it pitch-class segment} is an ordered subset of $\mathbb{Z}_{12}$, or more generally $\mathbb{Z}_m$. We use parentheses\footnote{We do not use the traditional musical notation $\langle x_1, \dots , x_n \rangle$ for pitch-class segments because it clashes with the mathematical notation for the subgroup generated by $x_1, \dots, x_n$, which we will also need on occasion.} to denote a pitch-class segment as an $n$-tuple $( x_1, \dots , x_n)$.  The sequential order of the pitch classes may, for example, relate to the temporal order of notes in a score, or to the distribution of pitches in different voices in a certain registral order. In connection with recent studies to voice leading, such as \cite{Call_Quinn_Tym}, one may wish to include voice permutations into the investigation of contextual transformations in non-trivial ways, as we do in Section~\ref{sec:Permutations}.

\subsection{Lewinian Duality between the $T/I$-Group and $PLR$-Group} \label{subsec:Duality_TI_PLR}
The $T/I$-group consists of the 24 bijections $T_j, I_j\co \mathbb{Z}_{12} \to \mathbb{Z}_{12}$ with $T_j(k)=k+j$ and $I_j(k)=-k+j$, where $j \in \mathbb{Z}_{12}$. Via its componentwise action on $3$-tuples, this dihedral group acts simply transitively on the set $S$ of all the transposed and inverted forms of the root position $C$-major 3-tuple $(0,4,7)$. Note that the minor triads in $S$ are not in root position, e.g. $a$-minor is $(4,0,9)$. Like any group action, this action corresponds to a homomorphism from the group to the symmetric group on the set upon which it acts, namely a homomorphism $\lambda \co T/I \to \mathrm{Sym}(S)$. The {\it symmetric group on $S$}, denoted $\mathrm{Sym}(S)$, consists of all bijections $S \to S$, while the group homomorphism $\lambda \co T/I \to \mathrm{Sym}(S)$ is $g \mapsto ( s \mapsto gs )$.  Since the action is simply transitive, the homomorphism $\lambda$ is an embedding, and we consider the $T/I$-group as a subgroup of $\mathrm{Sym}(S)$ via this embedding $\lambda$.

The other key character in this by now classical story is the neo-Riemannian $PLR$-group, which is the subgroup of $\mathrm{Sym}(S)$ generated by the bijections $P, L, R \co S \to S$. These transformations, respectively called {\it parallel}, {\it leading-tone exchange}, and {\it relative}, and are given by\footnote{Our usage of ordered $n$-tuples allows these root-free, mathematical formulations of musical operations. See also \cite[Footnote 20]{fioresatyendra2005}.}
\begin{equation} \label{equ:PLR_root_position}
\aligned
P(y_1,y_2,y_3)&:=I_{y_1+y_3}(y_1,y_2,y_3)&=(y_3,-y_2+y_1+y_3,y_1) \phantom{.}\\
L(y_1,y_2,y_3)&:=I_{y_2+y_3}(y_1,y_2,y_3)&=(-y_1+y_2+y_3,y_3,y_2) \phantom{.} \\
R(y_1,y_2,y_3)&:=I_{y_1+y_2}(y_1,y_2,y_3)&=(y_2,y_1,-y_3+y_1+y_2).
\endaligned
\end{equation}
For instance, $P(0,4,7)=(7,3,0)$, $L(0,4,7)=(11,7,4)$, and $R(0,4,7)=(4,0,9)$. These operations are sometimes called {\it contextual inversions} because the inversion in the definition depends on the input.\footnote{For an approach to contextual inversions in terms of indexing functions and a choice of canonical representative, see Kochavi \cite{kochavi}.}  Note that input and output always have two pitch-classes in common, though {\it their positions are reversed}. In Example~\ref{examp:Cohn_group}, we will see how to use permutations to define variants $P', L', R' \co S' \to S'$ which {\it retain the positions of the common tones}, and generate a dihedral group of order 24 we call the {\it Cohn group}. We will also see in Section~\ref{sec:Permutations} how permutations allow us to mathematically extend $P$, $L$, and $R$ to triads in first inversion or second inversion. Note that a naive application of the formulas in \eqref{equ:PLR_root_position} to a major triad not in root position makes no musical sense; for instance, naive application would erroneously show that the parallel of a first inversion $C$-major chord $(4,7,0)$ is $I_{4+0}(4,7,0)=(0,9,4)$, which is an $a$-minor chord.

The main properties of the $PLR$-group were observed by David Lewin: it acts simply transitively on $S$, and it consists precisely of those elements of $\mathrm{Sym}(S)$ which commute with the $T/I$-group. For instance $RT_7(0,4,7)=(11,7,4)=T_7R(0,4,7)$.
\begin{mydefinition}[Dual groups in the sense of Lewin, see page 253 of \cite{LewinGMIT}]
Let $\text{\rm Sym}(S)$ be the symmetric group on the set $S$. Two subgroups $G$ and $H$ of the symmetric group $\text{\rm Sym}(S)$ are {\it dual in the sense of Lewin} if their natural actions on $S$ are simply transitive and each is the centralizer of the other, that is, $$C_{\text{\rm Sym}(S)}(G)=H \;\; \text{ and } \;\; C_{\text{\rm Sym}(S)}(H)=G.$$
\end{mydefinition}
For an exposition of $T/I$-$PLR$-duality, see Crans--Fiore--Satyendra \cite{cransfioresatyendra}, and for its extension to length $n$ pitch-class segments in $\mathbb{Z}_m$ satisfying a tritone condition, see Fiore--Satyendra \cite{fioresatyendra2005}. Childs and Gollin both developed the relevant dihedral groups in the special case of the pitch-class
segment $X=( 0,4,7, 10 )$, i.e., for the set class of
dominant-seventh chords and half-diminished seventh
chords (see \cite{childs} and \cite{gollin}).

\subsection{Construction of the Dual Group in the Finite Case after Fiore--Noll \cite{fiorenoll2011}}  \label{subsection:dual_construction} The dual group for a simply transitive action of a finite group always exists. This was pointed out in \cite{fiorenoll2011}, though not proved there, so we present a proof now. Let $S$ be a general finite set, as opposed to the specific set of pitch-class segments in Section~\ref{subsec:Duality_TI_PLR}.

\begin{proposition}[Construction 2.3 of Fiore--Noll \cite{fiorenoll2011}, Finite Case] \label{prop:dual_construction}
Suppose $G$ is a finite group which acts simply transitively on a finite set $S$. Fix an element $s_0 \in S$ and consider the two embeddings
$$
\begin{array}{c}
\xymatrix{\lambda \co G \ar[r] & \text{\rm Sym}(S)}\\
\xymatrix{ \text{\phantom{$\lambda \co $}} g \ar@{|->}[r] & \Big( s \mapsto gs\Big)}
\end{array}
\hspace{.25cm}
\begin{array}{c}
\xymatrix{\rho \co G \ar[r] & \text{\rm Sym}(S)} \\
\xymatrix@C=1.5pc{\quad \quad \quad \quad g \ar@{|->}[r] & \Big(hs_0 \mapsto hg^{-1}s_0\Big).}
\end{array}$$
Then the images $\lambda(G)$ and $\rho(G)$ are dual groups in $\text{\rm Sym}(S)$. The injection $\rho$ depends on the choice of $s_0$, but the image $\rho(G)$ does not.
\end{proposition}
\begin{proof}
If $j,k\in G$, then $\lambda(j)$ and $\rho(k)$ commute because $$\lambda(j)\rho(k)(hs_0)=j(hk^{-1})s_0=(jh)k^{-1}s_0=\rho(k)\lambda(j)(hs_0)$$ for any $h \in G$. Simple transitivity of both $\lambda(G)$ and $\rho(G)$ is fairly clear. Thus, so far we have $\rho(G) \subseteq C_{\mathrm{Sym}(S)}\big(\lambda(G)\big)$ and $\vert\rho(G)\vert=\vert G \vert = \vert S \vert$. Recall from the Orbit-Stabilizer Theorem that a finite group acting on a finite set acts simply if and only if it acts transitively, and in this case the cardinality of the group is the same as the cardinality of the set.

We next claim that the centralizer $C_{\mathrm{Sym}(S)}(\lambda(G))$ acts simply on $S$. If $c, c' \in C_{\mathrm{Sym}(S)}\big(\lambda(G)\big)$
and $cs_1=c's_1$ for some single $s_1 \in S$, then $chs_1=c'hs_1$ for all $h \in G$, which means $c$ and $c'$ are equal as functions on $S$. Thus this centralizer acts simply and $\vert C_{\mathrm{Sym}(S)}\big(\lambda(G)\big) \vert=\vert S \vert$, and consequently the inclusion $\rho(G) \subseteq C_{\mathrm{Sym}(S)}\big(\lambda(G)\big)$ from above is actually an equality. A similarly counting argument shows that
$\lambda(G) = C_{\mathrm{Sym}(S)}\big(\rho(G)\big)$.
\end{proof}

\begin{corollary}
If $S$ is a finite set, and a subgroup $G$ of $\mathrm{Sym}(S)$ acts simply transitively on $S$, then the centralizer of $G$ in $\mathrm{Sym}(S)$ also acts simply transitively.
\end{corollary}

We will use this construction several times in the following sections to find the dual group for the symmetric group $\Sigma_n$ acting on $n$-tuples and to include permutations into $T/I$-$PLR$-duality.

\section{Permutation Actions}
\label{sec:Permutations}

We now turn to the main theorem of this paper, Theorem~\ref{thm:dual_groups_with_permutations}.
Let $\Sigma_3$ denote the symmetric group on $\{1,2,3\}$. Its coordinate-permuting action on $3$-tuples in $\mathbb{Z}_{12}$
commutes with transposition and inversion. When we consider all transpositions and inversions of all reorderings of $(0,4,7)$, the $T/I$-group and symmetric group $\Sigma_3$ form an internal direct product denoted $\Sigma_3(T/I)$. Theorem~\ref{thm:dual_groups_with_permutations} essentially says in this case that the dual group to $\Sigma_3(T/I)$ is the internal direct product of the dual group to $\Sigma_3$ and the $PLR$-group, where $P$, $L$, and $R$ are defined on a reordering $\sigma(0,4,7)$ by $\sigma P \sigma^{-1}$, $\sigma L \sigma^{-1}$, and $\sigma R \sigma^{-1}$. Theorem~\ref{thm:dual_groups_with_permutations} is formulated more generally for $n$-tuples and any group of invertible affine maps instead of for 3-tuples and $T/I$. The method for constructing dual groups is always Proposition~\ref{prop:dual_construction}.

Of course, everything in this section works just as well for general $\mathbb{Z}_m$ beyond $\mathbb{Z}_{12}$, but we work with $\mathbb{Z}_{12}$ for concreteness.

\subsection{The Standard Permutation Action on $n$-tuples and its Dual Group} \label{subsection:Standard_Permutation_Action}
Let $\Sigma_n$ denote the symmetric group on $\{1,\dots,n\}$. Consider the standard left action of the symmetric group $\Sigma_n$ on all $n$-tuples with $\mathbb{Z}_{12}$ entries,
$$\xymatrix{\Sigma_n\times\left(\mathbb{Z}_{12}\right)^n \ar[r] & \left(\mathbb{Z}_{12}\right)^n}$$
defined\footnote{The inverses must be included because the first inclination to define $\sigma  (y_1, \dots, y_n )=\left(y_{\sigma(1)},\dots, y_{\sigma(n)}\right)$ is not a left action, since we would have $(\sigma \sigma')Y=\sigma'(\sigma Y)$.} by $\sigma  (y_1, \dots, y_n ):=\left(y_{\sigma^{-1}(1)},\dots, y_{\sigma^{-1}(n)}\right)$. Let $X = (x_1,  \dots, x_n)$ denote a particular pitch-class segment with $n$ {\it distinct} pitch classes, and consider its orbit
$$\Sigma_nX=\left\{ \left(x_{\sigma^{-1}(1)}, \dots, x_{\sigma^{-1}(n)}\right) \, | \, \sigma \in \Sigma_n \right\}.$$
This orbit $\Sigma_nX$ consists of all the reorderings of $X$, or
all the {\it permutations} of $X$. The restricted left action on the orbit
$$\xymatrix{\Sigma_n\times \left( \Sigma_nX \right) \ar[r] & \Sigma_nX}$$
is clearly simply transitive, as the components of $X$ are distinct.
Consequently, we have an associated embedding
$$\xymatrix{\lambda\co \Sigma_n\ar[r] &  \text{Sym}\left(\Sigma_nX\right)},$$
the image of which we call $\lambda(\Sigma_n)$.

As in Construction~2.3 of \cite{fiorenoll2011}, recalled in Section~\ref{subsection:dual_construction} above, we now construct the dual group $\rho(\Sigma_n)$ for $\lambda(\Sigma_n)$ in the symmetric group $\text{Sym}(\Sigma_nX)$. The fixed element $s_0$ is $X$. By simple transitivity, any element of $\Sigma_nX$ can be written as $\nu X$ for some unique $\nu \in \Sigma_n$. On the set of $X$-permutations $\Sigma_nX$, we define in terms of the standard left action a second left action
$$\xymatrix{\Sigma_n\times \left( \Sigma_nX \right) \ar[r]^-\cdot & \Sigma_nX}$$
 by $\sigma\cdot(\nu X):=(\nu \sigma^{-1})X$. One can quickly
check from the axioms for the standard left action that $$(\sigma \tau)\cdot(\nu X)=\sigma \cdot \left(\tau \cdot (\nu X)\right)$$
 $$e\cdot (\nu X)=\nu X$$
and that this second left action is simply transitive. This second
left action gives us a second embedding
$$\xymatrix{\rho\co \Sigma_n \ar[r] & \text{Sym}\left(\Sigma_nX\right)},$$
the image of which we call $\rho(\Sigma_n)$. The groups $\lambda(\Sigma_n)$ and $\rho(\Sigma_n)$ commute because
$$\sigma(\nu \tau^{-1})X=(\sigma \nu)\tau^{-1}X$$
for all $\sigma, \nu, \tau \in \Sigma_n$. We have sketched a proof of the following proposition (and by example also some details of Proposition~\ref{prop:dual_construction}).
\begin{proposition}
The order $n!$ groups $\lambda(\Sigma_n)$ and $\rho(\Sigma_n)$ are dual subgroups of $\text{\rm Sym}(\Sigma_nX)$,
which has order $(n!)!$\,.
\end{proposition}

\subsection{The Standard Permutation Action and its Dual Group in the Case $n=3$} \label{subsec:Permutations_n=3_Case}
The standard permutation action and its dual group in the case $n=3$ are of particular interest for our present paper. We now work out explicitly this special case of Section~\ref{subsection:Standard_Permutation_Action}. Let $X = (x_1, x_2, x_3)$ denote the pitch-class segment of a trichord. The symmetric group on 3 letters in cycle notation\footnote{We follow the standard cycle notation {\it without commas}. For example, the cycle $(123)$ is the map $1 \mapsto 2 \mapsto 3 \mapsto 1$. Cycles are composed as ordinary functions are. For example, $(123)(23)=(12)$ because we do $(23)$ first and then $(123)$.} is $$\Sigma_3 = \{\text{id}, (123), (132), (23), (13), (12)\}.$$
We obtain $$\Sigma_3X= \left
\{ \begin{array}{rll}
X      & = & (x_1, x_2, x_3) \\
(123)X & = & (x_3, x_1, x_2) \\
(132)X & = & (x_2, x_3, x_1) \\
(23)X  & = & (x_1, x_3, x_2) \\
(13)X  & = & (x_3, x_2, x_1) \\
(12)X  & = & (x_2, x_1, x_3)
\end{array} \right \}.$$ As generators for the actions $\lambda(\Sigma_3)$ and $\rho(\Sigma_3)$ we may choose $\lambda(1 2 3)$, $\lambda(2 3)$ and $\rho(1 2 3)$, $\rho(2 3)$, respectively, which have the following explicit form.
$$\begin{array}{ll}
\lambda(1 2 3):
\begin{array}{rrr}
X        & \mapsto & (1 2 3)X \\
(1 2 3)X & \mapsto & (1 3 2)X \\
(1 3 2)X & \mapsto & X \\
(2 3)X   & \mapsto & (1 2)X \\
(1 3)X   & \mapsto & (2 3)X \\
(1 2)X   & \mapsto & (1 3)X
\end{array},
&
\rho(1 2 3):
\begin{array}{rrr}
X         & \mapsto & (1 3 2)X \\
(1 2 3)X  & \mapsto & X \\
(1 3 2)X  & \mapsto & (1 2 3)X \\
(2 3)X    & \mapsto & (1 2)X \\
(1 3)X    & \mapsto & (2 3)X \\
(1 2)X    & \mapsto & (1 3)X
\end{array} \\
& \\
\lambda(2 3):
\begin{array}{rrr}
X        & \mapsto & (2 3)X \\
(1 2 3)X & \mapsto & (1 3)X \\
(1 3 2)X & \mapsto & (1 2)X \\
(2 3)X   & \mapsto & X \\
(1 3)X   & \mapsto & (1 2 3)X \\
(1 2)X   & \mapsto & (1 3 2)X
\end{array},
&
\rho(2 3):
\begin{array}{rrr}
X         & \mapsto & (2 3)X \\
(1 2 3)X  & \mapsto & (1 2)X \\
(1 3 2)X  & \mapsto & (1 3)X \\
(2 3)X    & \mapsto & X \\
(1 3)X    & \mapsto & (1 3 2)X \\
(1 2)X    & \mapsto & (1 2 3)X
\end{array}
\end{array}$$

We may write these generators more compactly in cycle notation.
$$\begin{array}{lll}
\lambda(1 2 3) & = &\Big(X\;\;\; (1 2 3)X\;\;\; (1 3 2)X \Big) \Big((2 3)X\;\;\; (1 2)X\;\;\; (1 3)X \Big) \\

\lambda(2 3) & = & \Big(X\;\;\; (2 3)X \Big) \Big((1 2 3)X\;\; \;(1 3)X \Big) \Big((1 3 2)X \;\;\; (1 2)X \Big)\\

\rho(1 2 3) & = &\Big (X \;\;\; (1 3 2)X \;\;\; (1 2 3)X \Big) \Big((2 3)X \;\;\; (12)X\;\;\; (1 3)X \Big)\\

\rho(2 3) & = & \Big(X\;\;\; (2 3)X \Big) \Big((1 2 3)X\;\;\; (1 2)X \Big) \Big((1 3 2)X\;\;\; (1 3)X \Big)
\end{array}$$

\subsection{Affine Groups with Permutations and their Duals}

Now consider a pitch-class segment $X = (x_1,  \dots, x_n)$ with $n$ distinct pitch classes $x_k$ and a group $G \subseteq\mathrm{Aff}^{\ast}(\mathbb{Z}_{12})$ of invertible affine transformations. We let $G$ act componentwise on
$n$-tuples, and consider the orbit $GX$ of $X$. We assume, for the sake
of simplicity, that the underlying set of $X$ is not symmetric with
respect to any element of $G$. That is,  we require $f\{x_1,  \dots, x_n\} \neq\{x_1,  \dots, x_n\}$ for all $f \in G$. This condition guarantees that $G$
acts simply transitively on $GX$ and that none of the affine transformations $f \in G$, except the identity transformation, acts on $X$ merely like a permutation. We now extend the action of $\Sigma_n$
on $\Sigma_nX$ to an action on $\Sigma_nGX$.

The group $\Sigma_nG=G\Sigma_n$ is the subgroup of $\text{Sym}\Big((\mathbb{Z}_{12})^n\Big)$ generated by $\Sigma_n$ and $G$. Since $\Sigma_n$ and $\mathrm{Aff}^{\ast}(\mathbb{Z}_{12})$ commute, the group $\Sigma_nG$ is an internal direct product of $\Sigma_n$ and $G$, and every element of $\Sigma_nG$ can be written uniquely as $\sigma g$ with $\sigma \in \Sigma_n$ and $g\in G$. Recall that a group $H$ is an {\it internal direct product of subgroups $K$ and $L$} if $K$ and $L$ commute, $K \cap L = \{e\}$, and every element of $G$ can be written as $k\ell$ for some $k \in K$ and $\ell \in L$.

The orbit of $X$ under $\Sigma_nG$ decomposes as a disjoint union, which gives a principle $\Sigma_n$-bundle over the pitch-class sets underlying the $G$-orbit of $X$.
$$G\Sigma_nX=\underset{g \in G}{\coprod} \Sigma_n(gX)  \xymatrix{ \ar[r] & }G\{x_1, \dots, x_n\} $$
As detailed in Section~\ref{subsection:Standard_Permutation_Action}, on each set $\Sigma_n(gX)$ in the disjoint union we have dual groups $\lambda^g(\Sigma_n)$ and $\rho^g(\Sigma_n)$ in $\text{Sym}(\Sigma_n(gX))$. In light of the disjoint union decomposition, these actions fit together to give commuting, but not dual,\footnote{These two groups cannot be dual, because they do not act simply transitively: their cardinalities are $n!$ while the set upon which they act has cardinality $|G|\cdot n!$.} subgroups of $\text{Sym}(G\Sigma_nX)$. However, these commuting groups form part of dual groups as in the following theorem.

\begin{theorem}[Affine Groups with Permutations and their Duals] \label{thm:dual_groups_with_permutations}
Let $X=(x_1, \dots, x_n)$ be a pitch-class segment in $\mathbb{Z}_{12}$ with $n$ distinct pitch-classes $x_1, \dots, x_n$.  Let $G$ be a subgroup of the group $\mathrm{Aff}^*(\mathbb{Z}_{12})$ of all invertible affine transformations $\mathbb{Z}_{12} \to \mathbb{Z}_{12}$, which acts componentwise on all $n$-tuples in $\mathbb{Z}_{12}$. Suppose $f\{x_1,  \dots, x_n\} \neq\{x_1,  \dots, x_n\}$ for all $f \in G$. Let $\Sigma_n$ denote the symmetric group on $n$ letters, which acts on $n$-tuples as in Section~\ref{subsection:Standard_Permutation_Action}. As above, let $\lambda(\Sigma_nG)$ be the subgroup of $\mathrm{Sym}(\Sigma_nG X)$ determined by the action of the internal direct product $\Sigma_nG$ on the orbit $\Sigma_nGX$. Recall that the dual group $\rho(\Sigma_nG)$ has elements $\rho(\nu h)$ for $\nu \in \Sigma_n$ and $h \in G$
where
$$\rho(\nu h)\sigma g X :=\sigma g \big( \nu h \big)^{-1} X$$
for all $\sigma \in \Sigma_n$ and $g \in G$.

Then:
\begin{enumerate}
\item \label{thm:dual_groups_with_permutations:i}
The restriction of the subgroup $\rho(\Sigma_n)$ to $\Sigma_nX$ is the dual group for $\lambda(\Sigma_n)$ in $\mathrm{Sym}(\Sigma_nX)$, and similarly the restriction of the subgroup $\rho(G)$ to
$GX$ is the dual group for $\lambda(G)$ in $\mathrm{Sym}(GX)$ .
\item \label{thm:dual_groups_with_permutations:ii}
The subgroups $\rho(\Sigma_n)$ and $\rho(G)$ of
$\mathrm{Sym}(\Sigma_nGX)$ commute, that is $\rho(\nu)\rho(h)=\rho(h)\rho(\nu)$ for all $\nu \in \Sigma_n$ and $h \in G$.
\item \label{thm:dual_groups_with_permutations:iii}
The group $\rho(\Sigma_nG)$ is the internal direct sum of $\rho(\Sigma_n)$ and $\rho(G)$.
\item \label{thm:dual_groups_with_permutations:iv}
If $Y \in \sigma GX$ and $h \in G$, then $\rho(h)Y=\sigma\rho(h)\sigma^{-1} Y$.
\end{enumerate}
\end{theorem}
\begin{proof}
Statement \ref{thm:dual_groups_with_permutations:i} follows directly from the construction of the dual group in Section~\ref{subsection:dual_construction}.
Statements \ref{thm:dual_groups_with_permutations:ii} and \ref{thm:dual_groups_with_permutations:iii} follow from the analogous facts about
$\Sigma_n$, $G$, and $\Sigma_nG$ because $\rho$ is an embedding (and consequently an isomorphism onto its image).
Alternatively, we may prove Statement \ref{thm:dual_groups_with_permutations:ii} as follows.
For $\nu \in \Sigma_n$ and $h \in G$ we have
$$
\begin{aligned}
\rho(\nu)\rho(h)\sigma g X & \overset{\text{def}}{=} \sigma g h^{-1} \nu^{-1} X \\
& = \sigma g \nu^{-1} h^{-1} X \\
& \overset{\text{def}}{=} \rho(h)\rho(\nu) \sigma g X,
\end{aligned}
$$
where the unlabeled equality follows from the fact that $\nu^{-1}$ and $h^{-1}$ commute because $\Sigma_n$ and $G$ commute as remarked above.
Statement \ref{thm:dual_groups_with_permutations:iv} follows from the fact that $\rho(h)$ commutes with $\sigma$ and $\sigma^{-1}$ by duality.
\end{proof}

\begin{myexample}[Permutations with $T/I$ and $PLR$-Duality] \label{examp:Permutations_with_T/I-PLR-Duality}
If in Theorem~\ref{thm:dual_groups_with_permutations} we take $X$ to be $(0,4,7)$ and $G$ to be the $T/I$-group, then we have the incorporation of permutations into $T/I$ and $PLR$-duality. In particular, $\Sigma_3 (T/I) (0,4,7)$ is the set of all possible orderings of major and minor triads, and
$\rho(\Sigma_3(T/I))$ is the internal direct product of $\rho(\Sigma_3)$ and the extended $PLR$-group. By part \ref{thm:dual_groups_with_permutations:iv} any operation $h$ of the $PLR$-group is extended to act on $Y =\sigma T_j(0,4,7)$ or $Y= \sigma I_j (0,4,7)$ by first translating back to ``root position,'' then operating, and then translating back, namely $hY:=\sigma h\sigma^{-1}Y$. For instance,
$$R(7,0,4)=(123)R(321)(123)(0,4,7)=(123)(4,0,9)=(9,4,0).$$
Another way to justify this is that the extended $R$ operation commutes with permutations, so
$$R(7,0,4)=R(123)(0,4,7)=(123)R(0,4,7)=(123)(4,0,9)=(9,4,0).$$
Thus, Theorem~\ref{thm:dual_groups_with_permutations}, in combination with the Sub Dual Group Theorem of Fiore--Noll \cite[Theorem~3.1]{fiorenoll2011}, gives a theoretical justification for the constructions at the end of  \cite[Section~5]{fiorenollsatyendraSchoenberg} concerning an analysis of Schoenberg, String Quartet Number 1, Opus 7.
\end{myexample}

\begin{myexample}[Cohn Group] \label{examp:Cohn_group}
We may now define new versions of $P$, $L$, and $R$ which retain the positions of common tones in the ordering of any triad.
Let $P':=\rho(13)P$, $L':= \rho(23)L$, and $R':=\rho(12)R$. Then we have for instance
$$L'(4,7,0)=\rho(23) L  (4,7,0) = L \rho(23) (321) (0,4,7)=$$
$$L (13) (0,4,7)=(13)L(0,4,7)=(13)(11,7,4)=(4,7,11)$$
by the table for $\rho(23)$ in Section~\ref{subsec:Permutations_n=3_Case}. See Figure~\ref{fig: CakeCognac} for further examples. We call the group generated by $P'$, $L'$, $R'$ the {\it Cohn group}. It is dihedral of order 24 (the relations can be checked directly using those of the $PLR$-group and the commutativity of $\rho(\Sigma_3)$ with the $PLR$-group).
\end{myexample}

\begin{myexample}[Venezia] Below we have a rhythmic reduction of Liszt, R. W. Venezia, S. 201, measures 31--42. The transformations in each of the three phrases are permutations, $P$, and $R$ operations, as pictured in the rows of the subsequent network. The vertical arrows of the network indicate that the three phrases are related by transposition by 3 semitones. All the squares commute by Theorem~\ref{thm:dual_groups_with_permutations}, since the four groups $\lambda(\Sigma_n)$, $\lambda(T/I)$, $\rho(\Sigma_n)$, and $\rho(T/I)=PLR$ commute.
\begin{center}
\includegraphics[width=13cm]{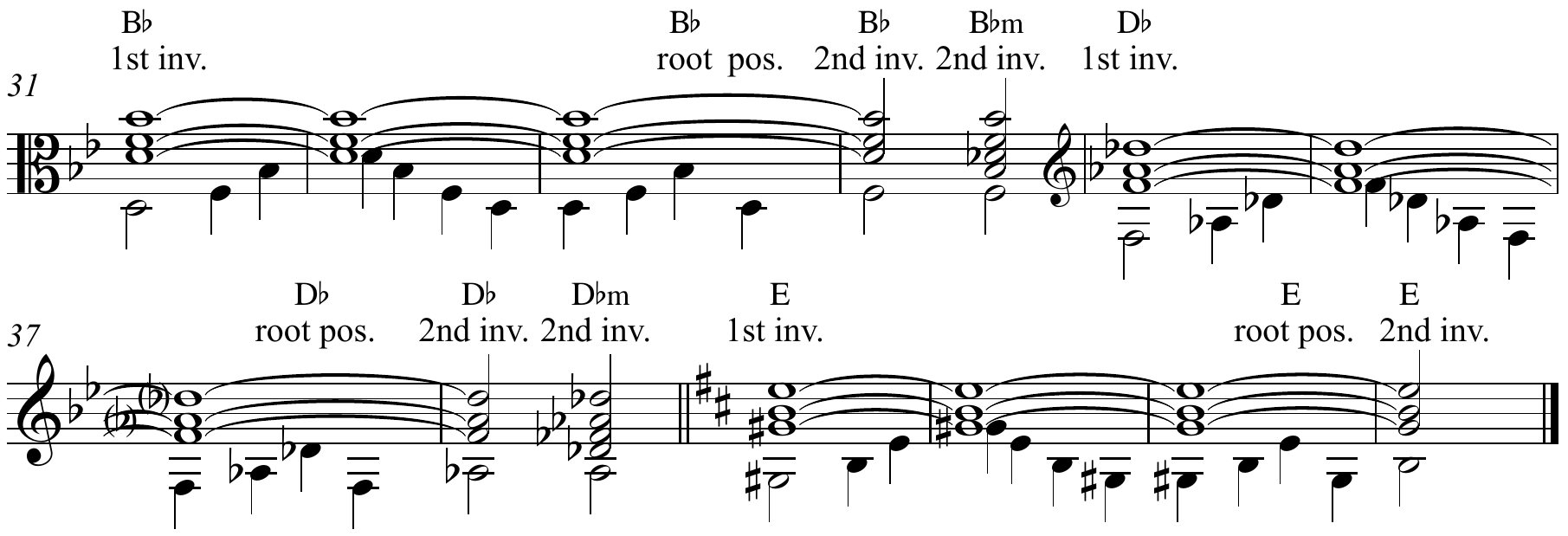}
\end{center}
{\footnotesize
$$
\renewcommand{\labelstyle}{\textstyle}
\entrymodifiers={+[F-]}
\xymatrix@C=3.4pc@R=2.8pc{\txt{$(2,5,10)$ \\ $B\flat$ 1st inv} \ar[r]^-{(123)} \ar[d]_{T_3} & \txt{$(10,2,5)$ \\ $B\flat$ root pos} \ar[r]^-{(13)}  \ar[d]_{T_3} & \txt{$(5,2,10)$ \\
$B\flat$ 2nd inv op} \ar[r]^-{(123)P} \ar[d]_{T_3} & \txt{$(5,10,1)$ \\ $B\flat m$ 2nd inv} \ar[r]^-{(13)R} \ar[d]_{T_3} & \txt{$(5,8,1)$ \\ $D\flat$ 1st inv} \ar[d]_{T_3}
\\
\txt{$(5,8,1)$ \\ $D\flat$ 1st inv} \ar[r]^-{(123)} \ar[d]_{T_3} & \txt{$(1,5,8)$ \\ $D\flat$ root pos} \ar[r]^-{(13)} \ar[d]_{T_3} & \txt{$(8,5,1)$ \\
$D\flat$ 2nd inv op} \ar[r]^-{(123)P} \ar[d]_{T_3} & \txt{$(8,1,4)$ \\ $D\flat m$ 2nd inv} \ar[r]^-{(13)R}  & \txt{$(8,11,4)$ \\ $E$ 1st inv}
\\
\txt{$(8,11,4)$ \\ $E$ 1st inv} \ar[r]^-{(123)} & \txt{$(4,8,11)$ \\ $E$ root pos} \ar[r]^-{(13)}  & \txt{$(11,8,4)$ \\
$E$ 2nd inv op} & *{} & *{} }$$
}

\end{myexample}

\section{Properties of Other Contextual Transformations on Pitch-Class Segments not Contained in $\rho(\Sigma_3(T/I))=\rho(\Sigma_3)PLR$}

The remainder of this paper illustrates some properties of {\it contextual inversion enchaining transformations}. These are certain transformations on pitch-class segments not contained in the dual group $\rho(\Sigma_n(T/I))$. In particular, we will discuss the $\mathrm{RICH}$-transformation, which goes beyond the scope of simply transitive actions as well as beyond the orbifold-construction via voice-permutation.

Consider the situation and notation of Theorem~\ref{thm:dual_groups_with_permutations}, and for $1 \leq q, r \leq n$ consider the globally defined {\it contextual inversion}\footnote{As we remarked earlier, the formulas in  equation \eqref{equ:PLR_root_position} for $P$, $L$, and $R$ are only valid for major triads in root position, or minor triads in the ordering $I_n(0,4,7)$. For other orderings of consonant triads, conjugation must be used, as in Example~\ref{examp:Permutations_with_T/I-PLR-Duality}. Thus, $J^{1,3}$, $J^{2,3}$, and $J^{1,2}$ do not agree with $P$, $L$, and $R$ beyond the $T/I$-class of $(0,4,7)$, and the name ``contextual inversion'' for $J^{q,r}$ it not optimal.}
\begin{equation} \label{equ:contextual_inversion}
J^{q,r}(Y):= I_{y_q+y_r}Y.
\end{equation}
Composites of contextual inversions with permutations yield instances of {\it contextual inversion enchaining transformations}. Within the symmetric group $\Sigma_n$, consider the order 2 cycle\footnote{Of course, an order 2 cycle is more commonly called a ``transposition'' in the mathematics literature, but we avoid using that term here because ``transposition'' already has other meanings in this article.} $(r \; s)$. On pitch-class segments $(y_1, \dots , y_n)$, the permutation $(r \; s)$ acts through voice exchange by mutually exchanging the pitch classes $y_r$ and $y_s$ at their respective positions in $(y_1, \dots, y_n)$. $$(r \; s)\co (y_1, \dots , y_r, \dots, y_s, \dots, y_n) \mapsto (y_1, \dots , y_s, \dots, y_r, \dots, y_n)$$

\begin{mydefinition}
Consider a pitch-class segment $X = (x_1, \dots, x_n)$ and select three distinct indices $1 \le q, r, s \le n$. A {\it contextual inversion enchaining transformation}
is any composite
$$(r \; s) \circ J^{q, r}\co \Sigma_n(T/I)X \to \Sigma_n(T/I)X$$ of a contextual inversion $J^{q, r}$ and a voice exchange $(r \; s)$ sharing the common index $r$.
\end{mydefinition}

The effect of enchaining will be illustrated by example. For $n = 3$ the cycle $(1\;3)$ behaves like a retrograde, which motivates Lewin's notation $\mathrm{RICH}$ in \cite{LewinGMIT} for the transformation $(1\;3) \circ J^{2, 3}$, meaning {\it retrograde inversion enchaining}. If $Y$ is a pitch-class segment, then $\text{RICH}(Y)$ is that retrograde inversion of $Y$ which has the first two notes $y_2$ and $y_3$, in that order.
This transformation was used in our analysis of Schoenberg in \cite{fiorenollsatyendraSchoenberg}.

The explicit cycle notation of the $\mathrm{RICH}$ transformation on consonant triads is displayed in Figure~\ref{fig:RICH_in_cycles}.
\begin{figure}
\caption{Cycle decomposition of $\mathrm{RICH}$ action on all 144 permutations of the major and minor triads}
\label{fig:RICH_in_cycles}
\begin{tabular}{|c|l|}
\hline {\bf Type} & {\bf Consonant Triad Cycles for $\mathrm{RICH}$} \\
\hline \hline
$RL$ & (4, 7, 11) (7, 11, 2) (11, 2, 6) (2, 6, 9) (6, 9, 1) (9, 1, 4) (1, 4,  8) (4, 8, 11) \\
& (8, 11, 3) (11, 3, 6) (3, 6, 10) (6, 10, 1) (10, 1, 5) (1, 5, 8) (5, 8, 0) (8, 0, 3) \\
& (0, 3, 7) (3, 7, 10) (7, 10, 2) (10, 2, 5) (2, 5, 9) (5, 9, 0) (9, 0, 4) (0, 4, 7)\\
\hline
$RL$ & (4, 0, 9) (0, 9, 5) (9, 5, 2) (5, 2, 10) (2, 10, 7) (10, 7, 3) (7, 3, 0) (3, 0, 8) \\
& (0, 8, 5) (8, 5, 1) (5, 1, 10) (1, 10, 6) (10, 6, 3) (6, 3, 11) (3, 11, 8) (11, 8, 4) \\
& (8, 4, 1) (4, 1, 9) (1, 9, 6) (9, 6, 2) (6, 2, 11) (2, 11, 7) (11, 7, 4) (7, 4, 0) \\
\hline \hline
$PR$ &  (0, 7, 3) (7, 3, 10) (3, 10, 6) (10, 6, 1) (6, 1, 9) (1, 9, 4) (9, 4, 0) (4, 0, 7)\\
\hline
$PR$ & (0, 4, 9) (4, 9, 1) (9, 1, 6) (1, 6, 10) (6, 10, 3) (10, 3, 7) (3, 7, 0) (7, 0, 4)\\
\hline
$PR$ & (1, 8, 4) (8, 4, 11) (4, 11, 7) (11, 7, 2) (7, 2, 10) (2, 10, 5) (10, 5, 1) (5, 1, 8)\\
\hline
$PR$ & (1, 5, 10) (5, 10, 2) (10, 2, 7) (2, 7, 11) (7, 11, 4) (11, 4, 8) (4, 8, 1) (8, 1, 5)\\
\hline
$PR$ & (2, 9, 5) (9, 5, 0) (5, 0, 8) (0, 8, 3) (8, 3, 11) (3, 11, 6) (11, 6, 2) (6, 2, 9)\\
\hline
$PR$ & (2, 6, 11) (6, 11, 3) (11, 3, 8) (3, 8, 0) (8, 0, 5) (0, 5, 9) (5, 9, 2) (9, 2, 6)\\
\hline \hline
$PL$ & (7, 4, 11) (4, 11, 8) (11, 8, 3) (8, 3, 0) (3, 0, 7) (0, 7, 4)\\
\hline
$PL$ & (7, 0, 3) (0, 3, 8) (3, 8, 11) (8, 11, 4) (11, 4, 7) (4, 7, 0)\\
\hline
$PL$ & (8, 5, 0) (5, 0, 9) (0, 9, 4) (9, 4, 1) (4, 1, 8) (1, 8, 5)\\
\hline
$PL$ & (8, 1, 4) (1, 4, 9) (4, 9, 0) (9, 0, 5) (0, 5, 8) (5, 8, 1)\\
\hline
$PL$ & (9, 6, 1) (6, 1, 10) (1, 10, 5) (10, 5, 2) (5, 2, 9) (2, 9, 6)\\
\hline
$PL$ & (9, 2, 5) (2, 5, 10) (5, 10, 1) (10, 1, 6) (1, 6, 9) (6, 9, 2)\\
\hline
$PL$ & (10, 7, 2) (7, 2, 11) (2, 11, 6) (11, 6, 3) (6, 3, 10) (3, 10, 7)\\
\hline
$PL$ & (10, 3, 6) (3, 6, 11) (6, 11, 2) (11, 2, 7) (2, 7, 10) (7, 10, 3) \\
\hline
\end{tabular}
\end{figure}
More specifically, in Theorem~\ref{thm:dual_groups_with_permutations}, we take $X$ to be $(0,4,7)$ and $G$ to be the $T/I$-group, so that $\Sigma_3(T/I)(0,4,7)$ is the $144=6\times 24$ possible orderings of major and minor triads, and $\rho(\Sigma_3(T/I))$ is the internal direct product of $\rho(\Sigma_3)$ and the $PLR$-group. The group $\rho(\Sigma_3(T/I))$ is also the subgroup of $\mathrm{Sym}(\Sigma_3(T/I))$ generated by $\rho(\Sigma_3)$ and the $PLR$-group. But $\mathrm{RICH}$ is not in the simply transitive group $\rho(\Sigma_3(T/I))$ as we now explain.

A close look at the cycle decomposition of $\mathrm{RICH}$ shows that there are cycles of length 24, behaving like $RL$-cycles, cycles of length 8, behaving like $PR$-cycles, and cycles of length 6, behaving like $PL$-cycles. Consequently the sixth and eighth powers $\mathrm{RICH}^6$ and $\mathrm{RICH}^8$ have fixed points, and $\mathrm{RICH}$ cannot be part of a simply transitive group action on all 144 ordered triads. In application to suitable subsets of $\Sigma_3(T/I)X$, e.g. to selected pitch-class segments in an octatonic cycle, the fixed-point effect disappears, and $\mathrm{RICH}$ can be part of a simply transitive group action on those. The first two $PR$-cycles in Figure~\ref{fig:RICH_in_cycles} involve 16 triadic pitch-class segments over the octatonic scale $\{0, 2, 3, 4, 6, 7, 9, 10\}$.

\smallskip

\begin{tabular}{|c|l|}
\hline
$PR$ & (0, 7, 3) (7, 3, 10) (3, 10, 6) (10, 6, 1) (6, 1, 9) (1, 9, 4) (9, 4, 0) (4, 0, 7)
\\ \hline
$PR$ & (1, 6, 10) (6, 10, 3) (10, 3, 7) (3, 7, 0) (7, 0, 4) (0, 4, 9) (4, 9, 1) (9, 1, 6) \\
\hline
\end{tabular}
\smallskip

\noindent The second one is precisely the $PR$-cycle in measures 88--92 of Schoenberg, String Quartet Number 1, Opus 7 pictured in
\cite[Figures~1 and 2]{fiorenollsatyendraSchoenberg}. This octatonically restricted $\mathrm{RICH}$-transformation involves two (and only two) Flip-Flop Cycles of length $8$ in the sense of John Clough \cite{cloughFlipFlop}. Analogous orbits can be obtained for pitch-class segments of jet and shark triads in \cite{fiorenollsatyendraSchoenberg}. The last $PR$-cycle in Figure~\ref{fig:RICH_in_cycles} contains the cello motive in measures 8--10, which is pictured in \cite[Figures~12 and 13]{fiorenollsatyendraSchoenberg},
and located in the octatonic scale
$\{2,3,5,6,8,9,11,0\}.$
See also the Summary Network in \cite[Figure~14]{fiorenollsatyendraSchoenberg}.


\end{document}